\documentclass[11pt]{article}
\usepackage{amsmath}
\usepackage{amsfonts}
\usepackage{latexsym}
\usepackage{hyperref}
\usepackage{enumerate}
\usepackage{wasysym}
\usepackage{lipsum}
\usepackage{amsthm}
\usepackage{amssymb}
\usepackage{csquotes}
\usepackage{xcolor}
\usepackage{tikz}
\usepackage{amssymb,amsmath,amsthm}
\usepackage{amsfonts}
\usetikzlibrary{positioning,chains,fit,shapes,calc}

\newtheorem{theorem}{Theorem}[section] 

\newtheorem{lemma}[theorem]{Lemma}

\newtheorem{definition}[theorem]{Definition}

\newtheorem{construction}[theorem]{Construction}

\def\rc{\mathrm{rc}}

\title{Rainbow Connection for Complete Multipartite Graphs}

\author{
Igor Araujo \footnote{University of Illinois at Urbana-Champaign. Email: \texttt{igoraa2@illinois.edu}. Research partially supported by UIUC Campus Research Board RB 22000.} \and 
Kareem Benaissa \footnote{University of Illinois at Urbana-Champaign. Email: \texttt{kareemyusefben@gmail.com}.} \and 
Richard Bi \footnote{University of Illinois at Urbana-Champaign. Email: \texttt{rbi3@illinois.edu}.} \and 
Sean English \footnote{University of North Carolina Wilmington. Email: \texttt{englishs@uncw.edu}} \and 
Shengan Wu \footnote{University of Illinois at Urbana-Champaign. Email: \texttt{shengan2@illinois.edu}.} \and 
Pai Zheng \footnote{University of Illinois at Urbana-Champaign. Email: \texttt{paiz3@illinois.edu}.}}

\date{}

\begin{document}
	\maketitle
	\begin{abstract}    
			A path in an edge-colored graph is said to be rainbow if no color repeats on it. An edge-colored graph is said to be rainbow $k$-connected if every pair of vertices is connected by $k$ internally disjoint rainbow paths. The rainbow $k$-connection number $\mathrm{rc}_k(G)$ is the minimum number of colors $\ell$ such that there exists a coloring with $\ell$ colors that makes $G$ rainbow $k$-connected. Let $f(k,t)$ be the minimum integer such that every $t$-partite graph with part sizes at least $f(k,t)$ has $\mathrm{rc}_k(G) \le 4$ if $t=2$ and $\mathrm{rc}_k(G) \le 3$ if $t \ge 3$. Answering a question of Fujita, Liu and Magnant, we show that
		\[
		f(k,t) = \left\lceil \frac{2k}{t-1} \right\rceil
		\]
		for all $k\geq 2$, $t\geq 2$. We also give some conditions for which $\mathrm{rc}_k(G) \le 3$ if $t=2$ and $\mathrm{rc}_k(G) \le 2$ if $t \ge 3$.


	\end{abstract}

	\section{Introduction}

	Let $G$ be a (simple, undirected) graph with an edge coloring $c:E(G)\to [\ell]$ for some integer $\ell$. We say a subgraph $F\subseteq G$ is \textbf{rainbow} if every edge of $F$ receives a different color under $c$. We will say that $(G,c)$ is \textbf{rainbow connected} if there is a rainbow path that connects every pair of vertices in $G$.
	
	Rainbow connection was first introduced by Chartrand, Johns, MeKeon and Zhang~\cite{cjmz2008}, where the authors defined the \textbf{rainbow connection number} $\rc(G)$, which is the least integer $\ell$ such that there exists a coloring $c:E(G)\to [\ell]$ such that $(G,c)$ is rainbow-connected. The same set of authors extended this definition to include higher connectivity~\cite{cjmz2009}. In particular, we will call $(G,c)$ \textbf{rainbow $k$-connected} if every pair of vertices in $G$ is connected by $k$ pairwise internally disjoint rainbow paths. The \textbf{rainbow $k$-connection number}, $\rc_k(G)$ is the minimum choice of $\ell$ such that there exists an edge coloring $c:E(G)\to [\ell]$ where $(G,c)$ is rainbow $k$-connected. Note that for a graph $G$ to have any hope of being rainbow $k$-connected, the graph itself must be $k$-connected.
	
	In~\cite{cjmz2009}, the authors considered complete graphs and balanced complete bipartite graphs and showed that for any integer $k\geq 1$,
	\[
	\mathrm{rc}_k(K_{(k+1)^2})=2,
	\]
	and
	\[
	\mathrm{rc}_k(K_{k^2,k^2})=3.
	\]
	This led the authors to ask for the least integer $f=f(k)$ such that $\mathrm{rc}_k(K_f)=2$ and the least integer $g=g(k)$ such that $\mathrm{rc}_k(K_{g,g})=3$. Li and Sun showed that $f=O(k^{3/2})$~\cite{lisun1}, and later showed that for any $r\geq g(k)$, we have $\mathrm{rc}_k(K_{r,r})=3$~\cite{lisun2}. Fujita, Liu and Magnant provided an improved upper bound in the bipartite case, $g(k)\leq 2k+o(k)$ using the probabilistic method~\cite{FLM}. In addition, they also showed that for $t\geq 3$ and $r\geq \frac{2k}{t-1}+o(k)$, if $K$ is the balanced complete $t$-partite graph with all parts of size $r$, then
	\[
	rc_k(K)=2.
	\]
	As more work came out on the behavior of the rainbow $k$-connection numbers for complete graphs and balanced complete multipartite graphs, the unbalanced case started to become of interest as well.
	
	Chartrand et. al. asked if there existed a function $h(k)$ such that for all $s,t$ with $h(k)\leq s\leq t$, we have $\mathrm{rc}_k(K_{s,t})=2$~\cite{cjmz2009}, however there is no such function as $\mathrm{rc}(K_{s,3^s+1})=4$ for all $s\in\mathbb{N}$~\cite{cjmz2008}. Thus, we could not expect a lower bound on the part sizes of an unbalanced complete bipartite graph to be able to force the rainbow $k$-connection number to get down to $3$, as happens in the balanced case. Similarly, for complete $t$-partite graphs, if $n_1\leq n_2\leq \dots\leq n_t$, and $n_t>2^m$, where $m=\sum_{i=1}^{t-1}n_t$, then $\mathrm{rc}(K_{n_1,n_2,\dots,n_t})=3$~\cite{cjmz2008}. Thus, in the unbalanced complete multipartite setting, again we cannot expect a lower bound on the part sizes to force the rainbow connection number down to $2$. One could wonder how close we can get though.

	\begin{definition}
		Given a complete $t$-partite graph $K_{n_1, \dots, n_t}$, let
		$f(k, t)$ be the minimum part size of $K$ such that,
		\[
		\rc_k(K) \leq 
		\begin{cases}
			4, & \text{if } t = 2, \\
			3, & \text{if }t > 2,
		\end{cases}
		\] 
		if it exists.
	\end{definition}
	Fujita, Lui and Magnant asked if such a function $f(k,t)$ exists~\cite{FLM}. In this work, we answer this question in the affirmative, showing that $f(k,t)$ exists for all pairs $k,t\geq 2$, and further we determine this function exactly.
	
	\begin{theorem} For all $k,t\geq 2$,
		\[
		f(k, t) = \left\lceil\frac{2k}{t-1}\right\rceil.
		\]
	\end{theorem}

	We also explore some instances where complete bipartite and complete multipartite graphs have rainbow $k$-connection number $3$ and $2$, respectively. For more information about rainbow connectivity, we direct the reader to the dynamic survey by Li and Sun~\cite{lisun3}.
	
	\section{Upper bound on \texorpdfstring{$f(k,t)$}{2}}
	
	In this section, we show that $f(k,t) \le \lceil\frac{2k}{t-1}\rceil$ for every $k,t \ge 2$. We first show that $f(k,2)\le 2k$. This is, we prove the desired upper bound for bipartite graphs.
	
	\begin{lemma}
		Let $a,b,k\in \mathbb{N}$ with $a,b\geq 2k$. Then
		\[
		\rc_k(K_{a,b})\leq 4.
		\]
	\end{lemma}
	
	\begin{proof}
		Let $A$ and $B$ be the partite sets of $K_{a,b}$ with $|A|=a$ and $|B|=b$. Let $A_1\cup A_2=A$ and $B_1\cup B_2=B$ be partitions of $A$ and $B$ such that $|A_1|,|A_2|,|B_1|,|B_2|\geq k$. Let $\{a_{i,1},a_{i,2},\dots,a_{i,k}\}\subseteq A_i$ and $\{b_{i,1},b_{i,2},\dots,b_{i,k}\}\subseteq B_i$ be collections of $k$ distinct vertices for each $i\in [2]$.
		
		Let $c:E(K_{a, b}) \to [4]$ be defined as
		\[
		c(uv) = 
		\begin{cases}
			1, & \text{if } u \in A_1 \text{ and } v \in B_1. \\
			2, & \text{if } u \in A_1 \text{ and } v \in B_2. \\
			3, & \text{if } u \in A_2 \text{ and } v \in B_1. \\
			4, & \text{if } u \in A_2 \text{ and } v \in B_2. \\
		\end{cases}
		\]
		See Figure~\ref{figure K_{a,b} coloring} for a diagram of the coloring.
		
		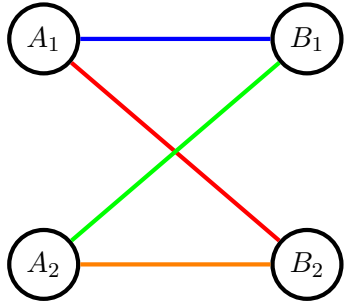
\begin{figure}
			\begin{center}
				\begin{tikzpicture}[ultra thick,
					every node/.style={draw,circle},
					every fit/.style={ellipse,draw,inner sep=-2pt,text width=2cm}]

					\begin{scope}[xshift=.25cm,yshift=3cm,start chain=going below,node distance=3mm]
						\foreach \i in {1}
						\node[on chain] (a\i)  {$A_1$};
					\end{scope}
					
					\begin{scope}[xshift=.25cm,yshift=0cm,start chain=going below,node distance=3mm]
						\foreach \i in {2}
						\node[on chain] (a\i)  {$A_2$};
					\end{scope}
					
					\begin{scope}[xshift=3.75cm,yshift=3 cm,start chain=going below,node distance=3mm]
						\foreach \i in {1}
						\node[on chain] (b\i)  {$B_1$};
					\end{scope}
					
					\begin{scope}[xshift=3.75cm,yshift=0cm,start chain=going below,node distance=3mm]
						\foreach \i in {2}
						\node[on chain] (b\i)  {$B_2$};
					\end{scope}
					
					\draw [blue] (a1) -- (b1);
					\draw [red] (a1) -- (b2);
					\draw [green] (a2) -- (b1);
					\draw [orange] (a2) -- (b2);
					
				\end{tikzpicture}
				\caption{A rainbow $k$-connected coloring of $K_{a,b}$ when $a,b\geq 2k$.}\label{figure K_{a,b} coloring}
			\end{center}
		\end{figure}
		
		We will now show that the pair $(K_{a, b}, c)$ is rainbow $k$-connected. Consider two arbitrary vertices, $u, v \in V(K_{a,b})$. Without loss of generality, assume that $u \in A_1$.
		
		\textbf{Case 1:} $v\in A_1$. Then the collection $\{(u,b_{1,j},a_{2,j},b_{2,j},v)\mid j\in [k]\}$ is a collection of $k$ pairwise internally disjoint rainbow paths from $u$ to $v$.
		
		\textbf{Case 2:} $v\in A_2$. Then the collection $\{(u,b_{1,j},v)\mid j\in [k]\}$ gives us a collection of $k$ pairwise internally disjoint rainbow paths from $u$ to $v$.
		
		\textbf{Case 3:} $v\in B$. We will assume without loss of generality that $v\in B_1$. Then $\{(u,b_{2,j},a_{2,j},v)\mid j\in [k]\}$ is a collection of $k$ pairwise internally disjoint rainbow paths from $u$ to $v$.
		
		Thus, $(K_{a,b},c)$ is rainbow $k$-connected.
	\end{proof}
	
	We now present a coloring that will be helpful for providing an upper bound on $f(k,t)$ when $t\geq 3$.
	
	\begin{construction}\label{construction c_{t,K}}
		Let $t,a_1,a_2,\dots,a_t\in\mathbb{N}$ with $t\geq 2$, and let $K:=K_{a_1,a_2,\dots,a_t}$.
		
		If $t$ is even, arbitrarily label the partite sets of $K$ by $A_1,A_2,\dots,A_{\frac{t}2},B_1,B_2,\dots,B_{\frac{t}2}$, and let $A=\bigcup_{i=1}^{\frac{t}2}A_i$ and $B=\bigcup_{i=1}^{\frac{t}2}B_i$. Then let $c_{t,K}:E(K)\to [3]$ be defined as follows.
		\[
		c_{t,K}(uv) = 
		\begin{cases}
			1, & \text{if } u,v\in A \text{ or }u,v \in B,\\
			2, &\text{if }u\in A_i\text{ and }v\in B_i\text{ for } 1\leq i\leq\frac{t}{2},\\
			3, & \text{if } u \in A_i \text{ and } v \in B_j \text { for }1 \leq i, j \leq \frac{t}{2}, i\neq j. \\
		\end{cases}
		\]
		
		If $t$ is odd, arbitrarily label the partite sets of $K$ by $X,A_1,A_2,\dots,A_{\frac{t-1}2},B_1,B_2,\dots,B_{\frac{t-1}2}$, and let $K'=K[V(K)\setminus X]$. Define $A=\bigcup_{i=1}^{\frac{t-1}{2}}A_i$ and $B=\bigcup_{i=1}^{\frac{t-1}{2}}B_i$. Then let $c_{t,K}:E(K)\to [3]$ be defined as follows.
		\[
		c_{t,K}(uv) = 
		\begin{cases}
			c_{t-1,K'}(uv) & \text{if } u,v\in V(K'),\\
			1, &\text{if }u\in A\text{ and }v\in X,\\
			3, & \text{if } u\in B \text{ and }v \in X.
		\end{cases}
		\]
		For an example of $c_{9,K}$, see Figure~\ref{figure c_{9,K}}.
	\end{construction}

	\begin{figure}
		\begin{center}
			\begin{tikzpicture}[ultra thick,
				every node/.style={draw,circle},
				every fit/.style={ellipse,draw,inner sep=-2pt,text width=2cm}]

				\begin{scope}[xshift=.25cm,yshift=1.5cm,start chain=going below,node distance=3mm]
					\foreach \i in {1}
					\node[on chain] (a\i)  {$A_1$};
				\end{scope}
				
				\begin{scope}[xshift=.25cm,yshift=0cm,start chain=going below,node distance=3mm]
					\foreach \i in {2}
					\node[on chain] (a\i)  {$A_2$};
				\end{scope}
				
				\begin{scope}[xshift=.25cm,yshift=-1.5cm,start chain=going below,node distance=3mm]
					\foreach \i in {3}
					\node[on chain] (a\i)  {$A_3$};
				\end{scope}
				
				\begin{scope}[xshift=.25cm,yshift=-3cm,start chain=going below,node distance=3mm]
					\foreach \i in {4}
					\node[on chain] (a\i)  {$A_4$};
				\end{scope}
				
				\begin{scope}[xshift=3.75cm,yshift=1.5 cm,start chain=going below,node distance=3mm]
					\foreach \i in {1}
					\node[on chain] (b\i)  {$B_1$};
				\end{scope}
				
				\begin{scope}[xshift=3.75cm,yshift=0cm,start chain=going below,node distance=3mm]
					\foreach \i in {2}
					\node[on chain] (b\i)  {$B_2$};
				\end{scope}
				
				\begin{scope}[xshift=3.75cm,yshift=-1.5cm,start chain=going below,node distance=3mm]
					\foreach \i in {3}
					\node[on chain] (b\i)  {$B_3$};
				\end{scope}
				
				\begin{scope}[xshift=3.75cm,yshift=-3cm,start chain=going below,node distance=3mm]
					\foreach \i in {4}
					\node[on chain] (b\i) {$B_{4}$};
				\end{scope}
				
				\begin{scope}[xshift=2cm,yshift=-5.5cm,start chain=going below,node distance=3mm]
					\foreach \i in {1}
					\node[on chain] (x\i) {$X$};
				\end{scope}
				
				\draw [red] (a1) -- (b1);
				\draw [green] (a1) -- (b2);
				\draw [green] (a1) -- (b3);
				\draw [green] (a1) -- (b4);
				
				\draw [green] (a2) -- (b1);
				\draw [red] (a2) -- (b2);
				\draw [green] (a2) -- (b3);
				\draw [green] (a2) -- (b4);
				
				\draw [green] (a3) -- (b1);
				\draw [green] (a3) -- (b2);
				\draw [red] (a3) -- (b3);
				\draw [green] (a3) -- (b4);
				
				\draw [green] (a4) -- (b1);
				\draw [green] (a4) -- (b2);
				\draw [green] (a4) -- (b3);
				\draw [red] (a4) -- (b4);
				
				\draw [blue]   (a1) to[out=-180,in=-180] (a4);
				\draw [blue]   (a1) to[out=-180,in=-180] (a2);
				\draw [blue]   (a1) to[out=-180,in=-180] (a3);
				\draw [blue]   (a2) to[out=-180,in=-180] (a3);
				\draw [blue]   (a2) to[out=-180,in=-180] (a4);
				\draw [blue]   (a3) to[out=-180,in=-180] (a4);
				
				\draw [blue]   (b1) to[out=0,in=0] (b4);
				\draw [blue]   (b1) to[out=0,in=0] (b2);
				\draw [blue]   (b1) to[out=0,in=0] (b3);
				\draw [blue]   (b2) to[out=0,in=0] (b4);
				\draw [blue]   (b2) to[out=0,in=0] (b3);
				\draw [blue]   (b3) to[out=0,in=0] (b4);
				
				\node[ dashed, fit= (a1) (a2) (a3) (a4), inner sep=-3 mm, thick] (allA) {};
				\node[ dashed, fit= (b1) (b2) (b3) (b4), inner sep=-3 mm, thick] (allB) {};
				
				\draw [blue]   (x1) to[out=180,in=-90] (allA);
				\draw [green]   (x1) to[out=0,in=-90] (allB);
				
			\end{tikzpicture}
			\caption{A $9$-partite graph $K$ with edges colored according to $c_{9,K}$}
			\label{figure c_{9,K}}
		\end{center}
	\end{figure}
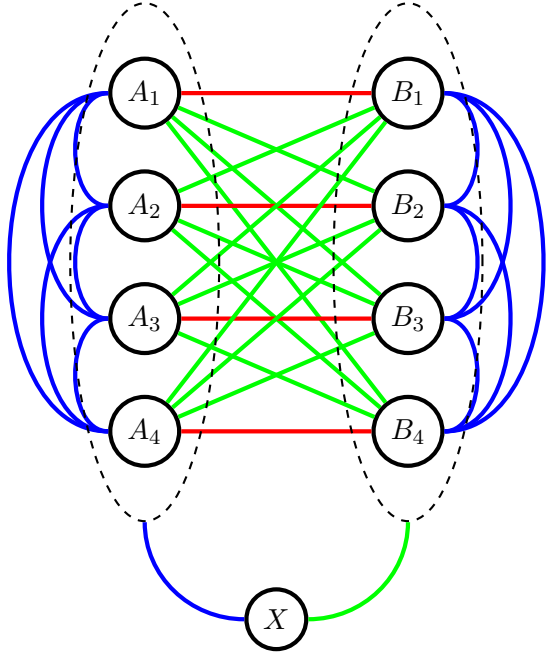
	
	We are now ready to give an upper bound when $t\ge 3$. We split the proof into two cases. Namely, when $t\ge 3$ is odd and when $t\ge 4$ is even.
	
	\begin{lemma}
		Let $k,t,a_1,a_2,\dots,a_t\in\mathbb{N}$ be such that $t\geq 3$ is odd and $\frac{2k}{t-1}\leq a_1\leq a_2\leq \dots\leq a_t$. Then
		\[
		\rc_k(K_{a_1,a_2,\dots,a_t})\leq 3.
		\]
	\end{lemma}
	
	\begin{proof}
		Let $s:=\left\lceil\frac{2k}{t-1}\right\rceil$ and let $K:=K_{a_1,a_2,\dots,a_t}$. Label the partite sets of $K$ and color the edges of $K$ according to $c:=c_{t,K}$, described in Construction~\ref{construction c_{t,K}}. For each $1\leq i\leq \frac{t-1}{2}$, let $\{a_{i,1},a_{i,2},\dots,a_{i,s}\}\subseteq A_i$ and  $\{b_{i,1},b_{i,2},\dots,b_{i,s}\}\subseteq B_i$ be collections of $s$ vertices. Similarly, let $\{x_1,x_2,\dots,x_s\}\subseteq X$ be a collection of $s$ vertices.
		
		We now show that $(K,c)$ is rainbow $k$-connected. Consider a pair, $u,v\in V(K)$.
		
		\textbf{Case 1:} At least one of $u$ or $v$ is in $A\cup B$. Assume without loss of generality that $u\in A_1$.
		
		\textbf{Case 1.1:} $v\in A_1$. Then 
		\[
		\left\{(u,a_{i,j},b_{i,j},v)\mid 2\leq i\leq \frac{t-1}{2}, 1\leq j\leq s\right\}\cup \{(u,x_j,b_{1,j},v)\mid 1\leq j\leq s\}
		\]
		forms a collection of $\frac{t-1}{2}s\geq k$ pairwise internally disjoint rainbow paths from $u$ to $v$.
		
		\textbf{Case 1.2:} $v\in A\setminus A_1$, assume without loss of generality that $v\in A_2$. Then 
		\[
		\left\{(u,a_{i,j},b_{i,j},v)\mid 3\leq i\leq \frac{t-1}{2}, 1\leq j\leq s\right\}\cup \{(u,x_j,b_{2,j},v)\mid 1\leq j\leq s\}\cup \{(u,b_{1,j},v)\mid 1\leq j\leq s\}
		\]
		give us $\frac{t-1}{2}s\geq k$ pairwise internally disjoint rainbow $u,v$-paths.
		
		\textbf{Case 1.3:} $v\in B$. Then 
		\[
		\left\{(u,a_{i,j},v)\mid 2\leq i\leq \frac{t-1}{2},1\leq j\leq s\right\}\cup \{(u,x_j,v)\mid 1\leq j\leq s\}
		\]
		contains $\frac{t-1}{2}s\geq k$ pairwise internally disjoint rainbow $u,v$-paths.
		
		\textbf{Case 1.4:} $v\in X$. Then
		\[
		\left\{(u,a_{i,j},b_{i,j},v)\mid 2\leq i\leq \frac{t-1}{2}, 1\leq j\leq s\right\}\cup \{(u,b_{1,j},v)\mid 1\leq j\leq s\}
		\]
		forms a collection of $\frac{t-1}{2}s\geq k$ pairwise internally disjoint rainbow $u,v$-paths.
		
		\textbf{Case 2:} Both $u,v\in X$. Then
		\[
		\left\{(u,a_{i,j},b_{i,j},v)\mid 1\leq i\leq \frac{t-1}{2}, 1\leq j\leq s\right\}
		\]
		give us a collection of $\frac{t-1}{2}s\geq k$ rainbow $u,v$-paths that are pairwise internally disjoint.
		
		In all cases, we find at least $k$ pairwise internally disjoint paths between any two vertices, so $\rc_k(K)\leq 3$.
	\end{proof}
	
	\begin{lemma}
		Let $k,t,a_1,a_2,\dots,a_t\in\mathbb{N}$ be such that $t\geq 4$ is even and $\frac{2k}{t-1}\leq a_1\leq a_2\leq \dots\leq a_t$. Then
		\[
		\rc_k(K_{a_1,a_2,\dots,a_t})\leq 3.
		\]
	\end{lemma}
	
	\begin{proof}
		Let $s:=\left\lceil\frac{2k}{t-1}\right\rceil$ and let $K:=K_{a_1,a_2,\dots,a_t}$. Label the partite sets of $K$ and color the edges of $K$ according to $c:=c_{t,K}$, described in Construction~\ref{construction c_{t,K}}. For each $1\leq i\leq \frac{t-1}{2}$, let $\{a_{i,1},a_{i,2},\dots,a_{i,s}\}\subseteq A_i$ and  $\{b_{i,1},b_{i,2},\dots,b_{i,s}\}\subseteq B_i$ be collections of $s$ vertices. Let $s_1:=\left\lceil \frac{s}2\right\rceil$ and $s_2:=\left\lfloor \frac{s}2\right\rfloor$, and note that $s_1+s_2=s$.
		
		We want to show that there exist at least $k$ pairwise internally disjoint rainbow paths between any two vertices in $K$, Let $u,v\in V(K)$, and assume without loss of generality that $u\in A_1$.
		
		\textbf{Case 1:} $v\in A_1$. Then
		\[
		\{(u,a_{2,j},b_{2,j},v)\mid 1\leq j\leq s_1\}\cup \{(u,a_{2,s_1+j},b_{1,s_1+j},v)\mid 1\leq j\leq s_2\}\cup\{(u,b_{1,j},b_{2,s_1+j}, v)\mid 1\leq j\leq s_2\}
		\]
		give us a collection of pairwise internally disjoint rainbow paths from $u$ to $v$, all of which only use vertices in $A_2\cup B_1\cup B_2\cup \{u,v\}$, see Figure~\ref{fig::evenupperboundcase1} for an example of such paths when $s=2$ and $t=4$. We also will add the paths in
		\[
		\left\{(u,a_{i,j},b_{i,j},v)\mid 3\leq i\leq \frac{t}2,1\leq j\leq s\right\}.
		\]
		\begin{figure}[!ht]
			\begin{center}
				\begin{tikzpicture}[thick,
					every node/.style={draw,ellipse},
					every fit/.style={ellipse,draw,inner sep=-2pt}]
					
					\begin{scope}[xshift=1cm,yshift=0cm, start chain=going below,node distance=7mm]
						\foreach \i in {u, v}
						\node[transform shape, fill=black, circle, draw, scale = 0.5,on chain] (a\i) [label=above: $\i$] {};
					\end{scope}
					
					\begin{scope}[xshift=1cm,yshift=-4.5cm, start chain=going below,node distance=7mm]
						\foreach \i in {1, 2}
						\node[transform shape, fill=black, circle, draw, scale = 0.5,on chain] (a\i) [label=above: $a_{2,\i}$] {};
					\end{scope}
					
					\begin{scope}[xshift=6cm,yshift=0cm, start chain=going below,node distance=7mm]
						\foreach \i in {1, 2}
						\node[transform shape, fill=black, circle, draw, scale = 0.5,on chain] (b\i) [label=above: $b_{1,\i}$] {};
					\end{scope}
					
					\begin{scope}[xshift=6cm,yshift=-4.5cm, start chain=going below,node distance=7mm]
						\foreach \i in {1, 2}
						\node[transform shape, fill=black, circle, draw, scale = 0.5,on chain] (b2\i) [label=above: $b_{2,\i}$] {};
					\end{scope}
					
					\draw (1,-0.5) ellipse (.6cm and 1.45cm);
					
					\draw (6,-0.5) ellipse (.6cm and 1.45cm);
					
					\draw (1,-5) ellipse (.6cm and 1.45cm);
					
					\draw (6,-5) ellipse (.6cm and 1.45cm);
					
					\draw[red, ultra thick] (au) -- (b1);
					\draw[green, ultra thick] (a1) -- (b1);
					\draw [blue, ultra thick]   (av) to[out=-180,in=-180] (a1);
					
					\draw[red, ultra thick] (au) -- (b2);
					\draw[green, ultra thick] (av) -- (b21);
					\draw [blue, ultra thick]   (b2) to[out=0,in=0] (b21);
					
					\draw[green, ultra thick] (au) -- (b22);
					\draw[red, ultra thick] (a2) -- (b22);
					\draw [blue, ultra thick]   (av) to[out=-180,in=-180] (a2);
					
					\node[draw=none] at (-0,0.5) {$A_1$};
					\node[draw=none] at (-0,-6) {$A_2$};
					\node[draw=none] at (7,0.5) {$B_1$};
					\node[draw=none] at (7,-6) {$B_2$};
					
				\end{tikzpicture}
			\end{center}
			\caption{Three pairwise internally disjoint rainbow $u,v$-paths on a 4-partite graph $K$ colored according to $c_{4,K}$.}
			\label{fig::evenupperboundcase1}
		\end{figure}
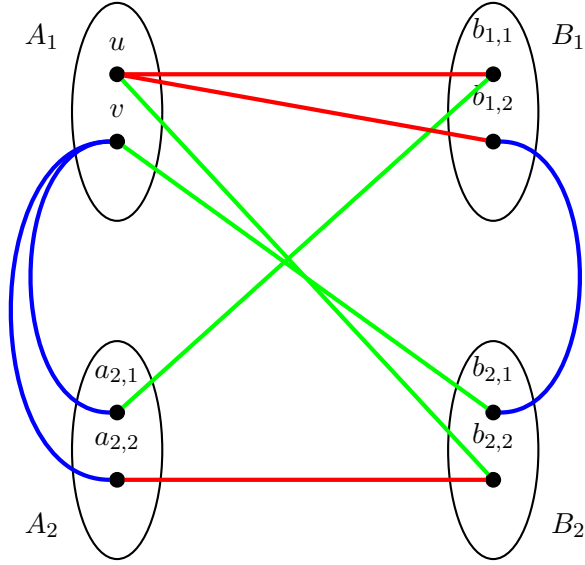
		
		We claim that the total number of paths counted above is at least $k$. Indeed, if $s$ is even (so $s_1=s/2$), then the collection of paths contains every vertex in 
		\begin{equation}\label{expression set of vertices in paths}
			\left\{a_{i,j}\mid 2\leq i\leq t/2,1\leq j\leq s\right\}\cup \{b_{i,j}\mid 1\leq i\leq t/2,1\leq j\leq s\}
		\end{equation}
		as an internal vertex, and if $s$ is odd, then the collection of paths contains every vertex in~\eqref{expression set of vertices in paths} except $b_{1,s_1}$, so in all cases we have at least $(t-1)s-1$ vertices internal to the paths, and each path is of length $3$, so it contains exactly two internal vertices, so the number of paths is at least
		\[
		\frac{(t-1)s-1}2=\frac{t-1}2\left\lceil\frac{2k}{t-1}\right\rceil-\frac{1}2\geq k-\frac{1}{2}.
		\]
		Now, since the number of paths and $k$ are both integers, the above inequality implies that there are at least $k$ paths.
		
		\textbf{Case 2:} $v\in A\setminus A_1$, assume without loss of generality that $v\in A_2$. Then
		\[
		\left\{(u,a_{i,j},b_{i,j},v)\mid 3\leq i\leq \frac{t}2, 1\leq j\leq s\right\}\cup\{(u,b_{i,j},v)\mid 1\leq i\leq 2, 1\leq j\leq s\}
		\]
		contains $\frac{t}2s\geq \frac{t}2\cdot \frac{2k}{t-1}>k$ pairwise internally disjoint rainbow $u,v$-paths.
		
		\textbf{Case 3:} $v\in B$. Let $i^*\in [t/2]$ be such that $v\not\in B_{i^*}$. There is such $i^*$ since $t \ge 4$. Then
		\[
		\{(u,a_{i,j},v)\mid 2\leq i\leq \frac{t}2,1\leq j\leq s\}\cup \{(u,b_{i^*,j},v)\mid 1\leq j\leq s\}
		\]
		contains $\frac{t}2s\geq k$ pairwise internally disjoint rainbow paths from $u$ to $v$.
		
		Thus, $(K,c)$ is rainbow $k$-connected, and so $\rc_k(K)\leq 3$.
	\end{proof}
	
	\section{Lower bound on \texorpdfstring{$f(k,t)$}{2}}
	
	In this section, we show that $f(k,t) \ge \lceil\frac{2k}{t-1}\rceil$ for every $k,t \ge 2$. As before, we split into cases. We first deal with bipartite graphs in the following lemma.
	
	\begin{lemma}
		Let $k,s,m\in\mathbb{N}$ with $k\leq s\leq 2k-1$ and $m\geq 4^{s}+1$. Then,
		\[
		\rc_k(K_{s,m})\geq 5.
		\]
	\end{lemma}
	
	\begin{proof}
		Let $K:=K_{s,m}$ and let $A$ and $B$ be the partite sets of $K$ with $|A|=m$. Assume to the contrary that $c:E(K)\to [4]$ is an edge coloring such that $(K,c)$ is rainbow $k$-connected. Since $c$ only uses four colors and $|A|>4^{|B|}$, by the pigeonhole principle there must exist two vertices $a_1,a_2\in A$ such that $c(a_1b)=c(a_2b)$ for all $b\in B$. This, along with the fact that $K$ is bipartite implies that every rainbow path from $a_1$ to $a_2$ must be of length $4$. Each such path contains two vertices in $B$, but $|B|<2k$, so $K$ does not contain $k$ pairwise internally disjoint paths of length $4$ from $a_1$ to $a_2$, a contradiction.
	\end{proof}
	
	We now generalize the previous argument for every $t\ge 3$ in the following.
	
	\begin{lemma}
		Let $k,t,s_1, s_2, \ldots, s_{t-1},m\in\mathbb{N}$ with $t\geq 3$, $\frac{k}{t-1}\leq s_i\leq \left\lceil\frac{2k}{t-1}\right\rceil-1$ for every $i \in [t-1]$ and $m\geq 3^{s_1+\dots+s_{t-1}}+1$. Let $K$ denote the complete $t$-partite graph with parts of size $m, s_1, s_2, \ldots, s_{t-1}$. Then,
		\[
		\rc_k(K)\geq 4.
		\]
	\end{lemma}
	
	\begin{proof}
		Let $A,B_1,\dots,B_{t-1}$ be the partite sets of $K$ with $|A|=m$, and $|B_i|=s_i$ for every $i \in [t-1]$. Let $B:=\bigcup_{i=1}^{t-1}B_i$. Assume to the contrary that there exists a coloring $c:E(K)\to [3]$ such that $(K,c)$ is rainbow $k$-connected. Since $c$ uses $3$ colors and $|A|>3^{|B|}$, by the pigeonhole principle there are two vertices $a_1,a_2\in A$ such that $c(a_1b)=c(a_2b)$ for all $b\in B$. Thus, every rainbow path from $a_1$ to $a_2$ must be of length $3$. Each path of length $3$ from $a_1$ to $a_2$ must have the internal vertex of the path in $B$. Now, $|B|=s_1+s_2+\ldots+s_{t-1}$, so there are at most $\frac{t-1}{2}\left(\left\lceil\frac{2k}{t-1}\right\rceil-1\right)$ pairwise internally disjoint paths of length $3$ from $a_1$ to $a_2$ in $K$. However,
		\[
		\frac{t-1}{2}\left(\left\lceil\frac{2k}{t-1}\right\rceil-1\right)< \frac{t-1}{2}\left(\frac{2k}{t-1}\right)=k,
		\]
		a contradiction.
	\end{proof}
	
	This concludes the proof that $f(k,t) = \lceil\frac{2k}{t-1}\rceil$ for every $k,t \ge 2$.
	
	\section{Complete multipartite graphs with very low rainbow \texorpdfstring{$k$}{2}-connec\-tion number}
	
	In this section, we find sufficient conditions for multipartite graphs to have rainbow $2$-connection number 2, the minimum possible value. We start with the following observation.
	
	\begin{theorem}
		Let $t,n_1,n_2,\dots,n_t\in\mathbb{N}$ with $t\geq 3$, and assume $\rc_2(K_{n_1, n_2,\dots, n_t}) = 2$. Then $\rc_2( K_{n_1+1,n_2+1,\dots,n_t} ) = 2$.
	\end{theorem}
	
	\begin{proof}
		Let $K:=K_{n_1+1,n_2+1,\dots,n_t}$ with partite sets $A_1\cup A_2\cup \dots\cup A_t=V(K)$, where $|A_1|=n_1+1$, $|A_2|=n_2+1$ and $|A_i|=n_i$ for all $3\leq i\leq t$. Let $a_1,a_1'\in A_1$ and $a_2,a_2'\in A_2$ be four distinct vertices, and let $K':=K[V(G)\setminus\{a_1,a_2\}]$. Note that $K'\cong K_{n_1, n_2,\dots, n_t}$, so $\rc_k(K')=2$ by assumption. Let $c':E(K')\to [2]$ be a coloring such that $(K',c')$ is rainbow $2$-connected. Assume without loss of generality that $c'(a_1'a_2')=1$. Define $c:E(K)\to [2]$ by
		\[
		c(uv)=\begin{cases} c'(uv)&\text{ if }u,v\in V(K'),\\
			c'(a_i'v)&\text{ if } u=a_i\text{ and }v\in V(K')\text{ for some }i\in [2]\\
			1&\text{ if }uv=a_1a_2,\\
			2&\text{ if }uv=a_1a_2'\text{ or }uv=a_1'a_2.
		\end{cases}
		\]
		We claim that $(K,c)$ is rainbow $2$-connected. Indeed, first note that since $(K',c')$ is rainbow $2$-connected and the restriction of $c$ down to $K'$ is $c'$, any pair $u,v\in V(K')$ has $2$ pairwise internally disjoint rainbow paths connecting $u$ and $v$. Furthermore, $K'':=K[V(K)\setminus \{a_1',a_2'\}]\cong K'$, and by the way that $c$ is defined, if $c''$ is the restriction of $c$ down to $K''$, then $(K'',c'')\cong (K',c')$, so any pair $u,v\in V(K'')$ is rainbow $2$-connected. The only pairs of vertices that remain to check are $a_1a_1'$, $a_1a_2'$, $a_1'a_2$ and $a_2a_2'$.
		
		First consider $a_ia_i'$ where $i \in [2]$. The paths $(a_i, a_j, a_i')$ and $(a_i, a_j', a_i')$ are pairwise internally disjoint rainbow paths, where $j\neq i$ and $j \in [2]$. Now consider $a_ia_j'$, where $\{i,j\}=[2]$. Note that $K''':=K[V(K)\setminus \{a_i',a_j\}]\cong K'$, and if we define $c'''$ to be the restriction of $c$ to $K'''$, but with the edge $a_ia_j'$ recolored to $c(a_ia_j)$, then $(K''',c''')\cong (K',c')$, and so $a_i$ and $a_j'$ are connected by $2$ internally disjoint rainbow paths in $(K''',c''')$. Indeed, the only $a_i,a_j'$-path affected by the edge $a_ia_j'$ is the one edge path, which is rainbow regardless of the color of $a_ia_j'$, so the number of pairwise internally disjoint rainbow paths remains the same. 
	\end{proof}
	
	\begin{theorem}
		Let $m,n\in\mathbb{N}$ with $n\geq 2$ and $1 \leq m \leq 4^{\left\lfloor \frac{n}{2}\right\rfloor}$. Then,
		\[
		\rc_2(K_{m, n, n}) = 2.
		\]
	\end{theorem}
	
	\begin{proof}
		First, we provide some definitions which will be helpful for giving a coloring. Let $s:=\left\lfloor \frac{n}{2}\right\rfloor$. Let $\mathcal{B}$ denote the collection of bit strings of length $2s$, and let $\mathbf{b}\in \mathcal{B}$ denote the bit string with $s$ $1$'s followed by $s$ $0$'s. Let $A\subseteq \mathcal{B}$ be a subset of $\mathcal{B}$ with $|A|=m$, and such that $\mathbf{b}\in A$. Let $A=\{a_1,a_2,\dots,a_m\}$ with $a_1=\mathbf{b}$, and let $a_{i,j}$ denote the $j$th bit of $a_i$.
		
		Let $B=\{b_1,b_2,\dots,b_n\}$ and $C=\{c_1,c_2,\dots,c_n\}$ be sets with $n$ distinct elements. Furthermore, for each $i\in [s-1]$, let $B_i=\{b_{2i-1},b_{2i}\}$ and $C_i=\{c_{2i-1},c_{2i}\}$ be pairs of vertices, and if $n$ is even, let $B_s=\{b_{2s-1},b_{2s}\}$ and $C_s=\{c_{2s-1},c_{2s}\}$, otherwise if $n$ is odd, let $B_s=\{b_{2s-1},b_{2s},b_{2s+1}\}$ and $C_s=\{c_{2s-1},c_{2s},c_{2s+1}\}$. Now, let $K$ be the copy of $K_{m,n,n}$ with partite sets $A$, $B$ and $C$. We now define a coloring $c:E(K)\to \{0,1\}$. Let
		\[
		c(uv)=\begin{cases}
			0 &\text{ if }uv=b_ic_i\text{ for some }i\in [n],\\
			1 &\text{ if }uv=b_ic_j\text{ for }i\neq j, i,j\in [n],\\
			a_{i,t} & \text{ if }uv=a_ib_j\text{ and }b_j\in B_t\text{ for }i\in [m], j\in [n],t\in [s],\\
			a_{i,s+t} &\text{ if }uv=a_ic_j\text{ and }c_j\in C_t\text{ for }i\in [m], j\in [n],t\in [s].
		\end{cases}
		\]
		
		We claim that $(K,c)$ is rainbow $2$-connected. Indeed, consider an arbitrary pair of vertices $u,v\in V(K)$.
		
		\textbf{Case 1:} $u,v\in B$ or $u,v\in C$, assume without loss of generality that $u,v\in B$, say $u=b_i$ and $v=b_j$. Then $(b_i,c_i,b_j)$ and $(b_i,c_j,b_j)$ are internally disjoint rainbow paths from $u$ to $v$.
		
		\textbf{Case 2:} $u\in B$ and $v\in C$, say $u=b_i$ and $v=c_j$. Then $(b_i,c_j)$ and $(b_i,a_1,c_j)$ are internally disjoint rainbow paths from $u$ to $v$.
		
		\textbf{Case 3:} $u\in A$ and $v\in B\cup C$, assume without loss of generality that $v\in B$, say $u=a_i$ and $v=b_j$. Then $(a_i,b_j)$ is one rainbow path. Let $t$ be such that $b_j\in B_t$, and let $c^*\in C_t\setminus\{c_j\}$. Then either $(a_i,c^*,b_j)$ or $(a_i,c_j,b_j)$ is a second rainbow path, depending on if $a_{i,s+t}$ is $0$ or $1$, respectively.
		
		\textbf{Case 4:} $u,v\in A$, say $u=a_i$ and $v=a_j$. Then there must be a value $t\in [2s]$ such that $a_{i,t}\neq a_{j,t}$. Then if $t\leq s$, $(a_i,b_{2t-1},a_j)$ and $(a_i,b_{2t},a_j)$ form a pair of internally disjoint rainbow paths, and if $t>s$, then $(a_i,c_{2(t-s)-1},a_j)$ and $(a_i,c_{2(t-s)},a_j)$ form a pair of internally disjoint rainbow paths.
		
		Thus, in all cases, every pair of vertices is connected by two internally disjoint rainbow paths, so $(K,c)$ is rainbow $2$-connected.
	\end{proof}
	
	\begin{theorem}\label{thm: 2-4-16}
		$\rc_2(K_{2,4,16}) = 2$.
	\end{theorem}
	
	\begin{proof}
		Let $A=\{a_1,a_2\}$, $B=\{b_1,b_2,b_3,b_4\}$ and $C=\{c_1,c_2,\dots,c_8,c_1',c_2',\dots,c_8'\}$ be the partite sets of $K:=K_{2,4,16}$. Let $\mathcal{B}=\{\mathbf{b}_1,\mathbf{b}_2,\dots,\mathbf{b}_8\}$ denote the collection of bit strings of length $4$ with an odd number of $0$'s, and let $b_{i,j}$ denote the value of the $j$th bit of $\mathbf{b}_i$.  We now define a coloring $c:E(K)\to \{0,1\}$. Let
		\[
		c(uv)=\begin{cases}
			0 &\text{ if }uv=a_ib_j\text{ for some }i\in [2],j\in[4]\\
			0 &\text{ if }uv=a_1c_j\text{ or }uv=a_2c_j'\text{ for }j\in [8],\\
			1 &\text{ if }uv=a_1c_j'\text{ or }uv=a_2c_j\text{ for }j\in [8],\\
			b_{i,j} & \text{ if }uv=b_jc_i\text{ or }uv=b_jc_i'\text{ for }i\in [8], j\in [4].
		\end{cases}
		\]
		
		We claim that $(K,c)$ is rainbow $2$-connected. Indeed, consider an arbitrary pair of vertices $x,y\in V(K)$.
		
	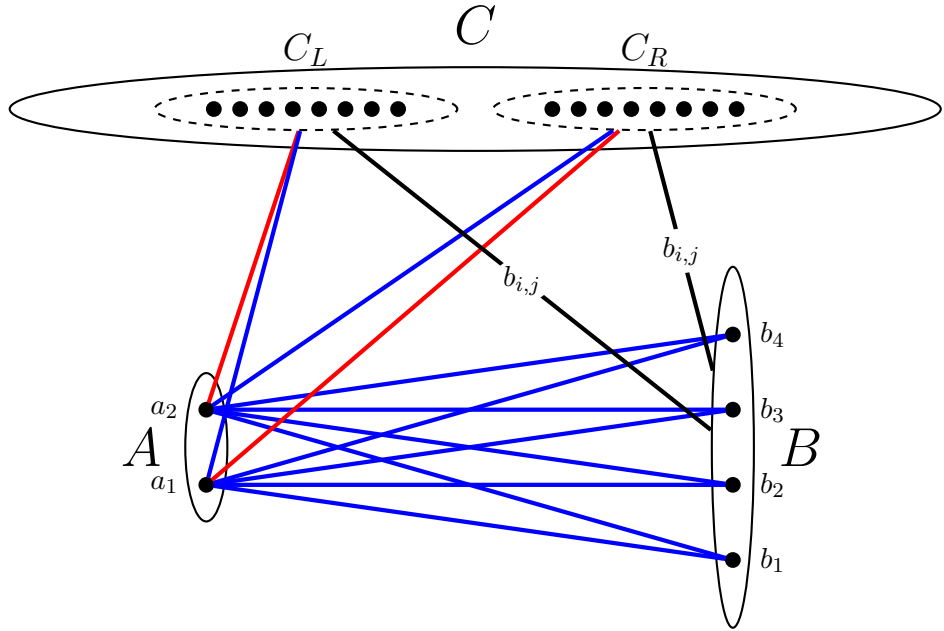
\begin{figure}[!b]
		\begin{center}
			\begin{tikzpicture}[thick,
				every node/.style={draw,ellipse},
				every fit/.style={ellipse,draw,inner sep=-2pt}]
				\begin{scope}[yshift = 2 cm]
					\begin{scope}[xshift = -2.25 cm]
						\foreach \i in {1,2,3,4,5,6,7,8}
						\node[transform shape, fill=black, circle, draw, scale = 0.5] (c\i) at (0.35*\i, 0) {};
						\node[dashed, fit= (c1) (c2) (c3) (c4) (c5) (c6) (c7) (c8), inner sep = 1 mm, label = above: \Large $C_L$] (allC1) {};
					\end{scope}
					\begin{scope} [xshift = 2.25 cm]
						\foreach \i in {1,2,3,4,5,6,7,8}
						\node[transform shape, fill=black, circle, draw, scale = 0.5] (d\i) at (0.35*\i, 0) {};
						\node[dashed, fit= (d1) (d2) (d3) (d4) (d5) (d6) (d7) (d8), inner sep = 1 mm, label = above: \Large $C_R$] (allD) {};
					\end{scope}
					\node[fit= (allC1) (allD), inner sep = 1 mm, label = above: \huge $C$] (allC) {};
				\end{scope}

				\begin{scope} [xshift = 5 cm, yshift = -5 cm]
					\foreach \i in {1,2,3,4}
					\node[fill=black, circle, draw, scale = 0.5, label = right: $b_{\i}$] (b\i) at (0, \i) {};
					\node[fit= (b1) (b2) (b3) (b4), inner sep = 1 mm, label=right: \huge $B$] (allB) {};
				\end{scope}
				\begin{scope} [xshift = -2 cm, yshift = -4 cm]
					\foreach \i in {1,2}
					\node[fill=black, circle, draw, scale = 0.5, label = left: $a_{\i}$] (a\i) at (0,\i) {};
					\node[ fit= (a1) (a2), inner sep = 1 mm, label = left: \huge $A$] (allA) {};
				\end{scope}
				
				\draw [blue, ultra thick] (a1) to (b1);
				\draw [blue, ultra thick] (a1) to (b2);
				\draw [blue, ultra thick] (a1) to (b3);
				\draw [blue, ultra thick] (a1) to (b4);
				
				\draw [blue, ultra thick] (a2) to (b1);
				\draw [blue, ultra thick] (a2) to (b2);
				\draw [blue, ultra thick] (a2) to (b3);
				\draw [blue, ultra thick] (a2) to (b4);
				
				\draw [red, ultra thick] (a2) to (allC1);
				\draw [red, ultra thick] (a1) to (allD);
				\draw [blue, ultra thick] (a1) to (allC1);
				\draw [blue, ultra thick] (a2) to (allD);
				\draw  (allC1)[black, ultra thick] -- node[draw = white, fill = white, inner sep = 0 mm] {$b_{i,j}$} (allB) ;
				\draw  (allD)[black, ultra thick] -- node[draw = white, fill = white, inner sep = 0 mm] {$b_{i,j}$} (allB) ;
			\end{tikzpicture}
		\end{center}
		\caption{The coloring of $K_{2,4,16}$ used in the proof of Theorem~\ref{thm: 2-4-16}. Here $C_L$ and $C_R$ represent the sets $\{c_1, \ldots, c_8\}$ and $\{c_1', \ldots, c_8'\}$, respectively.}
		\label{fig::2-4-16}
	\end{figure}	
		
		\textbf{Case 1:} $\{x,y\}=\{a_1,a_2\}$. Then $(x,c_1,y)$ and $(x,c_2,y)$ are pairwise internally disjoint rainbow paths.
		
		\textbf{Case 2:} $x\in A$ and $y\in B\cup C$ or $x\in B\cup C$ and $y\in A$, assume without loss of generality that $x\in A$, say $x=a_1$. First assume $y\in C$, say $y=c_i$. Note that for each $i\in [8]$, there exists a $j\in [4]$ such that $c_ib_j$ is color $1$, so $(x,y)$ and $(x,b_j,y)$ are pairwise internally disjoint rainbow paths. Now, if $y\in B$, say $y=b_i$, then note that for every $i\in [4]$ there exists a $j\in [8]$ such that $b_ic_j$ is color $1$. Hence, $(x,y)$ and $(x,c_j,y)$ are pairwise internally disjoint rainbow paths.
		
		\textbf{Case 3:} $x\in B$ and $y\in C$ or $x\in C$ and $y\in B$, assume $x\in B$. Then $(x,y)$ along with either $(x,a_1,y)$ or $(x,a_2,y)$ are pairwise internally disjoint rainbow paths.
		
		\textbf{Case 4:} $\{x,y\}\subseteq B$. Assume by symmetry that $x=b_1$ and $y=b_2$. Let $\mathbf{b}_i=0111$. Then $(x,c_i,y)$ and $(x,c_i',y)$ are pairwise internally disjoint rainbow paths.
		
		\textbf{Case 5:} $x=c_i$ and $y=c_j'$ or $x=c_i'$ and $y=c_j$ for some $i,j\in[8]$ (possibly with $i=j)$. Then $(x,a_1,y)$ and $(x,a_2,y)$ are pairwise internally disjoint rainbow paths.
		
		\textbf{Case 6:} $x=c_i$ and $y=c_j$ or $x=c_i'$ and $y=c_j'$ for some $i,j\in [8]$, assume $x=c_i$ and $y=c_j$. Note that $\mathbf{b}_i$ and $\mathbf{b}_j$ differ in at least two bits by construction, say $b_{i,s}\neq b_{j,s}$ and $b_{i,s'}\neq b_{j,s'}$ for some $s,s'\in [4]$. Then $(x,b_s,y)$ and $(x,b_{s'},y)$ are pairwise internally disjoint rainbow paths.
		
		Thus, in all cases, $x$ and $y$ are connected by two pairwise internally disjoint rainbow paths.
	\end{proof}

 \section{Acknowledgements}
 The authors would like to thank the Illinois Geometry Lab for facilitating this research project. This material is based upon work supported by the National Science Foundation under Grant No. DMS-1449269. Any opinions, findings, and conclusions or recommendations expressed in this material are those of the authors and do not necessarily reflect the views of the National Science Foundation.

	\bibliographystyle{abbrv}

\end{document}